\documentclass{article}
\usepackage{amssymb,amsmath,amsthm}
\usepackage{mathrsfs}
\usepackage[plainpages=false]{hyperref}
\usepackage{url}
\usepackage{bbm}   
\usepackage{graphicx}
\usepackage[left=2.6cm,right=2.6cm]{geometry}
\usepackage{pstricks}
\usepackage{pst-func}
\usepackage{pst-plot}
\usepackage{pst-fill}
\usepackage{subfigure}
\usepackage{tensor}

\definecolor{myurlcolor}{rgb}{0,0,0.7}
\definecolor{myrefcolor}{rgb}{0.8,0,0}
\usepackage{hyperref}
\hypersetup{colorlinks,
linkcolor=myrefcolor,
citecolor=myurlcolor,
urlcolor=myurlcolor}

\usepackage[inline,marginclue,final,silent]{fixme}

\usepackage{paralist}

\usepackage[all]{xy}
\input xy
\xyoption{all}
\xyoption{arc}
\xyoption{line}

\newcommand{\beq}{\begin{displaymath}}
\newcommand{\eeq}{\end{displaymath}}
\newcommand{\beqn}{\begin{equation}}
\newcommand{\eeqn}{\end{equation}}
\newcommand{\beqa}{\begin{eqnarray*}}
\newcommand{\eeqa}{\end{eqnarray*}}
\newcommand{\beqna}{\begin{eqnarray}}
\newcommand{\eeqna}{\end{eqnarray}}

\newcommand{\eq}[1]{~(\ref{#1})}

\newcommand{\N}{\mathbb{N}}
\newcommand{\Z}{\mathbb{Z}}
\newcommand{\R}{\mathbb{R}}
\newcommand{\C}{\mathbb{C}}
\newcommand{\F}{\mathbb{F}}
\renewcommand{\H}{\mathcal{H}}
\newcommand{\B}{\mathcal{B}}

\newcommand{\id}{\mathrm{id}}
\newcommand{\ra}{\rightarrow}

\newcommand{\lra}{\longrightarrow}

\newcommand{\tr}{\mathrm{tr}}

\newtheorem{prop}{Proposition}[section]
\newtheorem{thm}[prop]{Theorem}
\newtheorem{defn}[prop]{Definition}
\newtheorem{cor}[prop]{Corollary}
\newtheorem{lem}[prop]{Lemma}

\theoremstyle{definition} 

\newtheorem{ex}[prop]{Example}
\newtheorem{rem}[prop]{Remark}

\title{Operator system structures on the unital direct sum of $C^*$-algebras}
\author{Tobias Fritz\\[.3cm]
\small{ICFO--Institut de Ciencies Fotoniques, Mediterranean Technology Park, 08860 Castelldefels (Barcelona), Spain}\\[.3cm]
\small{\texttt{tobias.fritz@icfo.es}}}

\begin{document}

\maketitle

\begin{abstract}
This work is motivated by R\u{a}dulescu's result~\cite{Radu} on the comparison of $C^*$-tensor norms on $C^*(\F_n)\otimes C^*(\F_n)$.

For unital $C^*$-algebras $A$ and $B$, there are natural inclusions of $A$ and $B$ into the unital free product $A\ast_1 B$, the maximal tensor product $A\otimes_{\max} B$, and the minimal tensor product $A\otimes_{\min} B$. These inclusions define three operator system structures on the sum $A+B$. Partly using ideas from quantum entanglement theory, we prove various interrelations between these three operator systems. As an application, the present results yield a significant improvement over R\u{a}dulescu's bound. At the same time, this tight comparison is so general that it cannot be regarded as evidence for the QWEP conjecture.
\end{abstract}

\section{Introduction}

Connes' embedding problem~\cite{Con,Cap} is one of the most important open problems in operator algebra theory. It asks whether any separable von Neumann algebra factor of type $II_1$ can be embedded into $R^\omega$, which is the ultrapower (with respect to an arbitrary free ultrafilter on $\N$) of $R$, the hyperfinite $II_1$ factor.

The importance of Connes' embedding problem manifests itself in the large number of equivalent formulations known. Among these range several reformulations due to Kirchberg~\cite{Kir,Oza}, for example the \emph{QWEP conjecture}, which asks whether every $C^*$-algebra is a \emph{Q}uotient of one having the \emph{W}eak \emph{E}xpectation \emph{P}roperty. More recently, also questions in quantum information theory~\cite{HM,Jun,Fri} have been found to be equivalent. In this paper, we are going to consider the QWEP conjecture in the following formulation also due to Kirchberg: let $C^*(\F_n)$ be the full group $C^*$-algebra of $\F_n$, the free group on any finite number $n\geq 2$ of generators. Then the QWEP conjecture asks whether its $C^*$-tensor product with itself is unique:
\beqn
\label{qwep}
\textrm{QWEP conjecture: }\quad C^*(\F_n)\otimes_{\max} C^*(\F_n)\stackrel{?}{=}C^*(\F_n)\otimes_{\min} C^*(\F_n)
\eeqn
Recall~\cite{KR2} that one way to define the two sides of~(\ref{qwep}) is as the completions of the group algebra $\C[\F_n\times\F_n]$ with respect to the two norms
$$
||x||_{\max} = \sup_{\pi:\F_2\times\F_2\to\mathcal{U}(\H)} ||\pi(x)|| ,\qquad ||x||_{\min} = \sup_{\pi_a,\pi_b:\F_2\to\mathcal{U}(\H)} ||(\pi_a\otimes\pi_b)(x)||, \qquad x\in\C[\F_n\times\F_n],
$$
where the supremum ranges over all unitary representations of $\F_n\times\F_n$ and pairs of unitary representations of $\F_n$, respectively. With this in mind,~(\ref{qwep}) asks whether whether $||x||_{\max}\stackrel{?}{=}||x||_{\min}$ for every $x\in\C[\F_n\times\F_n]$, i.e.~whether unitary representations of $\F_n\times\F_n$ can be approximated by those of tensor product form in the appropriate sense.

If $U_1,\ldots,U_n$ stand for the generators of $\F_n$, then Pisier~\cite[p.6]{Pis2} has shown that it is sufficient to consider in $\C[\F_n\times\F_n]$ the linear span of all elementary tensors of the form
\beqn
\label{pisierradsub}
\mathbbm{1}\otimes\mathbbm{1},\quad U_i\otimes\mathbbm{1},\quad\mathbbm{1}\otimes U_j
\eeqn
in the following sense: the question~(\ref{qwep}) is equivalent to asking whether the two operator space structures induced from the embeddings into the two sides of~(\ref{qwep}) on this $(2n+1)$-dimensional subspace coincide completely isometrically. R\u{a}dulescu~\cite{Radu} has considered this formulation and deduced the following: whenever $X$ is any matrix with entries in this linear space, then
\beqn
\label{radres}
||X||_{\max}\leq\left(N^2-N\right)^{1/2}||X||_{\min}
\eeqn
with $N=2n+1$. In particular, the two operator space structures induced on this subspace from the two sides of~(\ref{qwep}) are completely isomorphic.

This paper follows this line of attack and establishes tight bounds on the comparison between $||X||_{\max}$ and $||X||_{\min}$. The recently found connection~\cite{Jun,Fri} between the QWEP conjecture and Tsirelson's problem on quantum correlations~\cite{SW} has allowed us to gain a certain physical intuition for~(\ref{qwep}). We use this intuition to derive, among other things, a rigorous improvement over~(\ref{radres}) which reduces the constant on the right-hand side to $2$, independently of $n$. For the precise statement, see section~\ref{qwepsec}.

In fact, the methods we use are completely general. We work with arbitrary unital $C^*$-algebras $A$ and $B$ and consider their ``unital direct sum'', by which we mean the linear space
\beq
A\otimes\mathbbm{1}+\mathbbm{1}\otimes B\subseteq A\otimes_{\mathrm{alg}} B
\eeq
equipped with the operator system structures induced on it by its natural embeddings into $A\otimes_{\min}B$, $A\otimes_{\max}B$ and the unital free product $A\ast_1 B$. (See definition~\ref{os} for the notion of operator system.) These three operator systems will be denoted by $A\oplus_{\min}B$, $A\oplus_{\max}B$ and $A\oplus_\ast B$, respectively. Following Pisier's trick~\cite{Pis2}, we can conclude that $A\otimes_{\min}B=A\otimes_{\max}B$ if and only if $A\oplus_{\min}B=A\oplus_{\max}B$ completely positively. We derive general results on the comparison between these three operator systems.

\paragraph{Summary.} Section~\ref{sec2} begins the main text by introducing the three operator systems $A\oplus_\ast B$, $A\oplus_{\max}B$ and $A\oplus_{\min}B$ which will be considered and contains derivations of some of their basic properties. Section~\ref{sec3} proceeds by generally constructing coproducts of operator systems, and shows that the coproduct of two unital $C^*$-algebras in the category of operator systems coincides with $A\oplus_{\ast}B$. This yields a characterization of the positive matrix cones of the operator system $A\oplus_\ast B$. Section~\ref{sec4} discusses some basic entanglement theory for states on tensor products of $C^*$-algebras. A quantum marginal problem is introduced there which is then applied to characterize the positive matrix cones of $A\oplus_{\max}B$ and $A\oplus_{\min}B$ in section~\ref{sec5}. Section~\ref{sec5} also contains the general statements and proofs of our main results, including a simple but powerful observation applicable to the comparison of norms on tensor products of $\Z_2$-graded $C^*$-algebras. Finally, section~\ref{qwepsec} then briefly states our results for the particular case $A=B=C^*(\F_n)$.

\paragraph{Notation and conventions.} Throughout the paper, $A$ and $B$ are arbitrary unital $C^*$-algebras. No separability assumption is needed. We always identify $A$ and $B$ with their images under the natural inclusions into a unital free product or a tensor product. For a unital $C^*$-algebra $A$, we write $\mathscr{S}(A)$ for its state space, and similarly $\mathscr{S}(S)$ for the state space of an operator system $S$. The topology used on $\mathscr{S}(A)$ is always going to be the weak $*$-topology. We abbreviate ``unital completely positive'' by ``ucp'' and $M_k(\C)$ by $M_k$. Often, it is better to think of a matrix algebra $M_k(A)$ as $M_k\otimes A$; these two notations will be used interchangeably. The minimal (``spatial'') and maximal tensor products of $C^*$-algebras are written as $A\otimes_{\min}B$ and $A\otimes_{\max}B$, respectively. We take $A\dot{\otimes}B$ to stand for some arbitrary $C^*$-algebraic tensor product.

While some familiarity with tensor products of $C^*$-algebras, operator systems and operator spaces is necessary for understanding of the results and their proofs, we have tried to include as much detail as possible. The main ideas are sufficiently simple that they should be understandable without an excessive background in functional analysis. We refer to~\cite{CE,Paul} for background on operator systems.

\paragraph{Acknowledgments.} The author would like to thank Edward Effros, Carlos Palazuelos and Vern Paulsen for helpful correspondence as well as Anthony Leverrier for discussion on entangled states in infinite-dimensional Hilbert spaces. This work has been supported by the EU STREP QCS.

\section{The unital direct sum of $C^*$-algebras as an operator system}
\label{sec2}

We begin with some basic observations in order to set the stage for the upcoming sections. So let $A$ and $B$ be unital $C^*$-algebras. Then there are natural embeddings of $A$ into
\begin{itemize}
\item the unital free product, $A\lra A\ast_1 B,\quad a\mapsto a$,
\item the maximal tensor product, $A\lra A\otimes_{\max}B.\quad a\mapsto a\otimes\mathbbm{1}$,
\item the minimal tensor product, $A\lra A\otimes_{\min}B,\quad a\mapsto a\otimes\mathbbm{1}$,
\end{itemize}
and likewise for $B$. The images of these maps are simply isomorphic copies of $A$ and $B$, respectively, and we will not distinguish between $A$ (resp.~$B$) and its isomorphic image. In each case, the linear hull of $A$ and $B$ together consists of the elements which are sums of the form $a+b$; in the tensor product case, such a sum actually stands for
\beq
a\otimes\mathbbm{1}+\mathbbm{1}\otimes b,
\eeq
but we will stick to the sloppy notation $a+b$. The splitting of $a+b$ into $a\in A$ and $b\in B$ is not unique, since any scalar can be added to $a$ and removed from $b$, while their sum stays the same. However, this is the only ambiguity which exists; more formally, we claim that all three sums $A+B$ are isomorphic, as $\C$-vector spaces, to the quotient vector space
\beqn
\label{uds}
\left(A\oplus B\right)/\C(\mathbbm{1}_A-\mathbbm{1}_B)\:,
\eeqn
which we also call a ``unital direct sum''. To see this, note that the units are identified as $\mathbbm{1}_A=\mathbbm{1}_B$ in each of the three cases, so that the sum $A+B$ is~(\ref{uds}) or some quotient thereof. But it cannot be a proper quotient since any pair of states $\alpha\in\mathscr{S}(A)$, $\beta\in\mathscr{S}(B)$ can be extended to a state on $A\ast_1 B$ (resp. $A\otimes_{\max}B$ or $A\otimes_{\min}B$), thereby separating everything else.

We will now consider this unital direct sum equipped with different structures of \emph{operator system}:

\begin{defn}[\cite{CE}]
\label{os}
Let $S$ be a complex vector space. $S$ is an \emph{operator system} if it comes equipped with a distinguished element $\mathbbm{1}\in S$ (the \emph{unit}), an antilinear involution $^*:S\to S$ with $\mathbbm{1}^*=\mathbbm{1}$, and, upon writing $M_k(S)_h=\{s\in M_k(S)\:|\: s^*=s\}$, a distinguished convex cone $M_k(S)_+\subseteq M_k(S)_h$ for every $k\in\N$ such that
\begin{enumerate}
\item The distinguished cones are salient: if $s\in M_k(S)_+$ and $-s\in M_k(S)_+$, then $s=0$,
\item For every $k\in\N$, the matrix $\mathbbm{1}_k=\mathrm{diag}(\mathbbm{1},\ldots,\mathbbm{1})\in M_k(S)$ is an \emph{order unit}: for every $s\in M_k(S)_h$, there is some $\lambda\in\R_{>0}$ with $(\lambda\mathbbm{1}_k-s)\in M_k(S)_+$,
\item For every $k\in\N$, the order unit $\mathbbm{1}_k$ is \emph{Archimedean}: if $(s+\lambda\mathbbm{1}_k)\in M_k(S)_+$ for all $\lambda\in\R_{>0}$, then $s\in M_k(S)_+$,
\item $\gamma^* M_k(S)_+ \gamma \subseteq M_l(S)_+$ for all $k,l\in\N$ and $\gamma\in M_{k,l}(\C)$.
\end{enumerate}
\end{defn}

For $s\in M_k(S)$, we also write $s\geq 0$ as synonymous to $s\in M_k(S)_+$. Every self-adjoint subspace of $\B(\H)$ naturally carries an operator system structure. It is a fundamental theorem of operator system theory~\cite{CE} that every operator system is of this form.

Every $C^*$-algebra is naturally an operator system, and this in particular applies to the three $C^*$-algebras mentioned above. Therefore, the unital direct sum~(\ref{uds}) inherits three operator system structures which we denote by, respectively,
\begin{align}
\label{opsyss}
\notag A\oplus_\ast B &\subset A\ast_1 B\:,\\
A\oplus_{\max} B &\subset A\otimes_{\max}B\:,\\
\notag A\oplus_{\min} B &\subset A\otimes_{\min}B\:.
\end{align}
Since~(\ref{uds}) is ``almost'' a direct sum, the most appropriate symbol for it seems to be a modified version of $\oplus$, which is what we chose. Contrary to what this notation may suggest, $A\oplus_{\max}B$ is not a maximal direct sum in whatever sense, and likewise for $A\oplus_{\min}B$. More generally, when $A\dot{\otimes}B$ is any $C^*$-tensor product of $A$ and $B$, we obtain an induced operator system structure on~(\ref{uds}) which is denoted by $A\dot{\oplus}B$. We will use this notation mainly for statements which apply to both $A\oplus_{\min}B$ and $A\oplus_{\max}B$.

Given operator systems $S$ and $T$, a \emph{ucp map} $\Phi:S\to T$ is a linear map which is unital, i.e.~$\Phi(\mathbbm{1}_S)=\mathbbm{1}_T$, and completely positive, i.e.~$\Phi(M_k(S)_+)\subseteq M_k(T)_+$ for all $k$ and with respect to the entrywise application of $\Phi$.

There are simple ucp relationships between the above operator systems. The natural inclusions $A\ra A\otimes_{\max}B$ and $B\ra A\otimes_{\max}B$ yield, by the universal property of $A\ast_1 B$, a $*$-homomorphism
\beq
A\ast_1 B\lra A\otimes_{\max}B\:.
\eeq
This is a quotient map which, morally, takes the unital free product and throws in relations stating that all elements of $A$ commute with those of $B$. Similarly, there is a natural projection
\beq
A\otimes_{\max}B\lra A\otimes_{\min}B\:.
\eeq
coming from the universal property of $A\otimes_{\max}B$.

Both these natural projections restrict to ucp maps between the operator system~(\ref{opsyss}). This can be concisely expressed in terms of the commutative diagram
\beqn
\begin{split}
\label{basic}
\xymatrix{ A\oplus_\ast B \ar[r]^(.45){\mathrm{ucp}} \ar@{^{(}->}[d] & A\oplus_{\max}B \ar[r]^{\mathrm{ucp}} \ar@{^{(}->}[d]  & A\oplus_{\min}B \ar@{^{(}->}[d] \\
A\ast_1 B \ar[r] & A\otimes_{\max}B \ar[r] & A\otimes_{\min}B}
\end{split}
\eeqn
which should always be kept in mind while reading the paper. Everything which follows from here on will be concerned with deriving further properties of these two ucp maps.

The following trick due to Pisier is an essential motivation for considering these operator systems:

\begin{prop}[\cite{Pis2}]
\label{onlyplus}
$A\oplus_{\max}B=A\oplus_{\min}B$ if and only if $A\otimes_{\min}B=A\otimes_{\max}B$.
\end{prop}

\begin{proof}
The ``if'' part is clear. For the converse, fix an embedding $A\otimes_{\max}B\subseteq\B(\H)$ and extend the canonical ucp map
\beq
A\oplus_{\min}B=A\oplus_{\max}B\lra A\otimes_{\max}B\subseteq\B(\H)
\eeq
to a ucp map $A\otimes_{\min}B\lra\B(\H)$ by the Arveson extension theorem. By construction, both $A$ and $B$ lie in the multiplicative domain~\cite{Choi},~\cite[2.7]{Oza} of this ucp map. Since $A$ and $B$ generate $A\otimes_{\min}B$ as a $C^*$-algebra, we conclude that the map $A\otimes_{\min}B\lra\B(\H)$ itself is multiplicative, and therefore is a $*$-homomorphism $A\otimes_{\min}B\lra A\otimes_{\max}B$ mapping $A$ to $A$ and $B$ to $B$.
\end{proof}

\section{Coproducts of operator systems and $C^*$-algebras}
\label{sec3}

We start this section with a detour on coproducts of operator systems. Throughout the following, we let $S$ and $T$ be arbitrary operator systems. Recall the following basic fact, which follows e.g.~from the representation theorem of operator systems:

\begin{rem}
\label{separate}
$s\in M_k(S)$ is positive if and only if $\phi(s)\geq 0$ for all states $\phi\in\mathscr{S}(M_k(S))$.
\end{rem}

Also for operator systems, it makes sense to consider the unital direct sum 
\beqn
\label{udsos}
S\oplus_1 T\equiv\left(S\oplus T\right)/\C\left(\mathbbm{1}_S-\mathbbm{1}_T\right)
\eeqn
which comes with natural inclusions of $S$ and $T$ and has as its unit the image of $\mathbbm{1}_S$ and $\mathbbm{1}_T$ under these inclusions. We will see soon that there is a certain operator system structure on this unital direct sum which extends those of $S$ and $T$ and is maximal with this property. (Here, we mean ``maximal'' in the sense that the positive matrix cones are as small as possible.) More formally, this unital direct sum will be a coproduct of $S$ and $T$ in the category of operator systems. But before getting to this, we need a definition.

\begin{defn}
Let $\phi\in\mathscr{S}(M_k(S))$ and $\chi\in\mathscr{S}(M_k(T))$ be states. We call $(\phi,\chi)$ a \emph{compatible pair} whenever $\phi_{|M_k}=\chi_{|M_k}$.
\end{defn}

In quantum-mechanical jargon~\cite{Hor}, $\phi$ is a bipartite state shared between a quantum system described by $S$ and an ancilla with state space $\C^k$. Similarly, $\chi$ is a bipartite state shared between a quantum system described by $T$ and the same ancilla. The states are compatible if and only if the reduced state of the ancilla is the same.

Note that for $k=1$, a pair of states is automatically compatible since necessarily $\phi(\mathbbm{1}_S)=1=\chi(\mathbbm{1}_T)$.

\begin{prop}
\label{coprod}
On the unital direct sum $S\oplus_1 T$, define an abstract operator system structure by stipulating that for any\footnote{In a notation like ``$s+t$'', it is understood that $s$ lies in $S$ (respectively $M_k(S)$) while $t$ is assumed to lie in $T$ (respectively $M_k(T)$).} $s+t\in M_k(S\oplus_1 T)$, the involution is given by $(s+t)^*=s^*+t^*$, and $s+t\geq 0$ holds if and only if
\begin{center}
$\phi(s)+\chi(t)\geq 0$ for all compatible pairs $\left(\phi\in\mathscr{S}(M_k(S)),\chi\in\mathscr{S}(M_k(T))\right)$\:.
\end{center}
With this definition and the natural inclusions $S\ra S\oplus_1 T$ and $T\ra S\oplus_1 T$, the operator system $S\oplus_1 T$ is the coproduct of $S$ and $T$ in the category of operator systems with ucp maps.
\end{prop}

\begin{proof}
It is clear that the thereby defined set of positive elements on $M_k(S\oplus_1 T)$ is a cone contained in the subspace of hermitian elements. Also, the cone is salient: any element $s+t$ which is contained in both the cone and its reflection at the origin necessarily satisfies $\phi(s)+\chi(t)=0$ for each pair of compatible matricial states. But since matricial states separate the elements of $M_k(S)$ and $M_k(T)$, it follows that $s=0$ and $t=0$. It can also be directly verified by the definition of positivity that $\mathbbm{1}$ is an Archimedean order unit of $M_k(S\oplus_1 T)$. The final but crucial operator system axiom concerns the compatibility of the different matricial cones with multiplication by scalar matrices. To this end, let $x\in M_{k,m}$ be a scalar matrix, $s+t\in M_k(S\oplus_1 T)$ be any positive element and $\phi\in\mathscr{S}(M_m(S))$ and $\chi\in\mathscr{S}(M_m(T))$ any pair of compatible states. Then it has to be shown that
\beq
\phi(x^*sx)+\chi(x^*tx)\geq 0\:
\eeq
To see this, note that $\phi(x^*x)=\chi(x^*x)$ by compatibility of the pair $(\phi,\chi)$, so that $\frac{\phi(x^*\cdot x)}{\phi(x^*x)}$ and $\frac{\chi(x^*\cdot x)}{\chi(x^*x)}$ are compatible states on $M_k(S)$ and $M_k(T)$, respectively. Hence the assertion follows from the assumption $s+t\geq 0$.

Now that it is clear that the operator system axioms are satisfied, it is time to check the coproduct property. By remark~\ref{separate}, the natural inclusions $S\ra S\oplus_1 T$ and $T\ra S\oplus_1 T$ are completely positive, while their unitality is also clear by definition. Now let $f:S\ra U$ and $g:T\ra U$ be any ucp maps to any other operator system $U$. By linearity, the only possible extension of $f$ and $g$ to $S\oplus_1 T$ is given by
\beq
S\oplus_1 T\lra U,\qquad s+t\mapsto f(s)+g(t)\:,
\eeq
so that it has to be verified that this is well-defined and ucp. By unitality of $f$ and $g$ and the definition~(\ref{udsos}), it is clearly well-defined and also unital. For complete positivity, consider any element $s+t\in M_k(S\oplus_1 T)$. Then again by remark~\ref{separate}, it is enough to take any state $\rho\in\mathscr{S}(M_k(U))$ and show that
\beq
\rho(f(s)+g(t))\geq 0.
\eeq
But this in turn is clear since $\rho\circ f$ is a state on $M_k(S)$ compatible with the state $\rho\circ g\in\mathscr{S}(M_k(T))$. This ends the proof.
\end{proof}

\begin{prop}
\label{scalarpos}
The positive matrix cones on the coproduct $S\oplus_1 T$ can also be characterized as follows: an element $s+t\in M_k(S\oplus_1 T)$ is positive if and only if
\beq
\exists\lambda\in M_k\quad\textrm{such that}\quad s-\lambda\geq 0\textrm{ in }S,\quad t+\lambda\geq 0\textrm{ in }T\:.
\eeq
\end{prop}

\begin{proof}
Again one can directly verify that this defines an abstract operator system structure together with ucp embeddings from $S$ and $T$. The universal property is clear. It coincides with the $S\oplus_1 T$ from the previous proposition by uniquness of the (isomorphism class of the) coproduct.
\end{proof}

Our next goal is to now take $S$ and $T$ to be $C^*$-algebras $A$ and $B$ (which carry a natural operator system structure), and then relate the coproduct $A\oplus_1 B$ to the operator systems~(\ref{opsyss}), in particular to $A\oplus_\ast B$.

As a preliminary note, recall that a representation of $M_k(\C)$ on some Hilbert space $\mathcal{H}$ is the same thing as a decomposition of $\mathcal{H}$ into $k$ orthogonal subspaces
\beq
\mathcal{H}_1=e_{11}\mathcal{H}\:,\quad\ldots\:,\quad\mathcal{H}_k=e_{kk}\mathcal{H}
\eeq
together with specified partial isometries which interpolate unitarily between any pair of these subspaces. Also it will frequently be used that an element of a $C^*$-algebra is positive if and only if its image under a representation of that $C^*$-algebra on a Hilbert spaces is always positive.

\begin{lem}
\label{equivmatrix}
Let $\xi_1,\ldots,\xi_k$ be a linearly independent set of vectors in some Hilbert space. Given another linearly independent set $\eta_1,\ldots,\eta_k$ such that
\beqn
\label{spc}
\langle\xi_i,\xi_j\rangle=\langle\eta_i,\eta_j\rangle\:,
\eeqn
then there is a unitary operator $U$ such that $\eta_i=U\xi_i$.
\end{lem}

\begin{proof}
It can be assumed w.l.o.g. that the Hilbert space has dimension $k$.

An application of Gram-Schmidt orthogonalization yields an invertible matrix $V_{ij}$ such that the vectors
\beq
\xi_i'\equiv\sum_j V_{ij}\xi_j
\eeq
form an orthonormal basis. But then also the vectors
\beq
\eta_i'\equiv\sum_j V_{ij}\eta_j
\eeq
form an orthonormal basis, and hence there is a unitary operator $U$ connecting the $\xi'$-basis to the $\eta'$-basis, and therefore also mapping the $\xi$-basis to the $\eta$-basis.
\end{proof}

\begin{prop}
\label{freepos}
For $a+b\in M_k(A\oplus_\ast B)$, the following are equivalent:
\begin{enumerate}
\item $a+b\geq 0$ in $M_k(A\ast_1 B)$,
\item If $\pi_A:A\ra\mathcal{B}(\mathcal{H})$ and $\pi_B:B\ra\mathcal{B}(\mathcal{H})$ are any two unital representations on the same Hilbert space, then $\pi_A(a)+\pi_B(b)$ acts as a positive operator on $\mathcal{H}^k=\mathcal{H}\otimes\C^k$.
\item For every compatible pair of states $\left(\alpha\in\mathscr{S}(M_k(A)), \beta\in\mathscr{S}(M_k(B))\right)$, it holds that $\alpha(a)+\beta(b)\geq 0$.
\item There is some $\lambda\in M_k$ such that $a-\lambda\geq 0$ and $b+\lambda\geq 0$.
\end{enumerate}
\end{prop}

\begin{proof}
The implication from (i) to (ii) is clear by the universal property of $A\ast_1 B$. Both (iii) and (iv) characterize the coproduct of $A$ and $B$ in the category of operator systems by the propositions~\ref{coprod} and~\ref{scalarpos}, hence these two are equivalent. Also (iv) clearly implies (i) since a sum of positive elements is positive.

It remains to prove the implication from (ii) to (iii). To this end, consider the GNS representations of some given $\alpha$ and $\beta$; we will use these to construct a Hilbert space $\mathcal{H}$ on which both $M_k(A)$ and $M_k(B)$ act. By the above remark on representations of matrix algebras, the Hilbert spaces of the two GNS representations are of the form $\mathcal{H}_\alpha\otimes\C^k$ and $\mathcal{H}_\beta\otimes\C^k$, where $A$ respectively $B$ only acts on the first factor, while $M_k$ acts on the second factor in the obvious way. The given GNS vectors are of the form
\beqn
\label{GNSvecs}
\sum_{i=1}^k\xi_i\otimes e_i\:\in\:\H_\alpha\otimes\C^k,\qquad\sum_{i=1}^k\eta_i\otimes e_i\:\in\H_\beta\otimes\C^k\:,
\eeqn
respectively, where\eq{spc} holds by compatibility of the pair $(\alpha,\beta)$. After possibly tensoring both by a third Hilbert space of big enough dimension, we find that $\mathcal{H}_\alpha$ and $\mathcal{H}_\beta$ are of the same dimension and therefore can be identified. This identification can be chosen to map the two GNS vectors\eq{GNSvecs} to each other by lemma~\ref{equivmatrix}. Now we have accomplished the situation that both states are induced from representations on the same Hilbert space $\H\otimes\C^k$ by the same vector, and the assumption (ii) can now be applied to yield the assertion (iii).
\end{proof}

In particular, we have therefore arrived at:

\begin{cor}
Let $A$ and $B$ be unital $C^*$-algebras. Their coproduct in the category of operator systems with ucp maps is given by
\beq
A\oplus_\ast B\equiv A+B\subseteq A\ast_1 B\:.
\eeq
\end{cor}

Furthermore, we can now also show that ucp maps on $C^*$-algebras can be extended to free products:

\begin{cor}
\label{extend}
Let $A$ and $B$ be unital $C^*$-algebras and 
\beq
\Phi:A\ra\B(\H),\qquad \Psi:B\ra\B(\H)
\eeq
ucp maps for some Hilbert space $\H$. Then there is a ucp map $\Omega:A\ast_1 B\ra\B(\H)$ extending these, i.e. such that the diagram
\beq
\xymatrix{ A \ar@{_{(}->}[d] \ar[rrd]^\Phi \\
 A\ast_1 B \ar@{-->}[rr]^(.4)\Omega && \B(\H) \\
B \ar@{^{(}->}[u] \ar[urr]_\Psi }
\eeq
commutes.
\end{cor}

\begin{proof}
This follows immediately from the present results by applying the Arveson extension theorem to the inclusion $A\oplus_\ast B\subseteq A\ast_1 B$.
\end{proof}

Corollary~\ref{extend} can be regarded as a weakened version of Boca's free products of ucp maps~\cite{Boca}.

\section{Basic entanglement theory for $C^*$-algebras}
\label{sec4}

This section is a digression with little relation to the rest of this work. The goal is to introduce the concepts and results necessary to prove proposition~\ref{ent}, which will then later be relevant for the proof of proposition~\ref{oncomp}. The definitions in this section are generalizations of the standard ones from quantum information theory~\cite{Hor}.

Since a $C^*$-algebra like $M_k(A\otimes_{\max}B)$ is actually a tensor product of \emph{three} $C^*$-algebras, the next few definitions and facts will be formulated generally for an arbitrary finite number of tensor product factors. So let $A_1,\ldots,A_n$ be unital $C^*$-algebras and
\beqn
\label{nprod}
A_1\dot{\otimes}\ldots\dot{\otimes}A_n
\eeqn
stand for any fixed $C^*$-algebraic tensor product. We now extend some basic concepts of quantum entanglement theory~\cite{Hor} to this setting.

\begin{defn}[product state]
Given states $\phi_i\in\mathscr{S}(A_i)$, we obtain a state $\phi_1\otimes\ldots\otimes\phi_n$ in $\mathscr{S}(A_1\dot{\otimes}\ldots\dot{\otimes}A_n)$ by setting
\beqn
\label{productstate}
\left(\phi_1\otimes\ldots\otimes\phi_n\right)\left(a_1\otimes\ldots\otimes a_n\right)\equiv\prod_i\phi_i(a_i),\qquad\forall a_i\in A_i
\eeqn
and extending by linearity and continuity. A state on $\mathscr{S}(A_1\dot{\otimes}\ldots\dot{\otimes}A_n)$ of this form is called a \emph{product state}.
\end{defn}

\begin{defn}[reduced state]
Given a state $\phi\in\mathscr{S}\left(A_1\dot{\otimes}\ldots\dot{\otimes}A_n\right)$, it restricts to a state in $\mathscr{S}(A_i)$ denoted by $\phi_{|A_i}$ as
\beq
\phi_{|A_i}(a)\equiv\phi\left(\mathbbm{1}\otimes\ldots\otimes a\otimes\ldots\otimes\mathbbm{1}\right)\:,
\eeq
where the $a$ stands at the $i$-th position in the product. We call $\phi_{|A_i}$ the \emph{reduced state of $\phi$ at $A_i$}.
\end{defn}

Similarly, one obtains a reduced state in $\mathscr{S}\left(\dot{\bigotimes}_{i\in\Omega}A_i\right)$ for each subset $\Omega\subseteq\{1,\ldots,n\}$. Obviously, taking reduced states of product states recovers the given states:
\beq
\left(\phi_1\otimes\ldots\otimes\phi_n\right)_{|A_i}=\phi_i\:.
\eeq

For the following, let it be reminded that we only consider the weak $*$-topology on state spaces, so this is what continuity refers to.

\begin{lem}
\label{cont}
\begin{enumerate}
\item Taking the reduced state is a continuous map
\beq
\mathscr{S}\left(A_1\dot{\otimes}\ldots\dot{\otimes}A_n\right)\lra\mathscr{S}(A_i)\:.
\eeq
\item The formation of product states as a map
\beq
\mathscr{S}(A_1)\times\ldots\times\mathscr{S}(A_n)\lra\mathscr{S}\left(A_1\dot{\otimes}\ldots\dot{\otimes}A_n\right)
\eeq
is continuous.
\end{enumerate}
\end{lem}

\begin{proof}
\begin{enumerate}
\item Clear by definition of the weak $*$-topology.
\item Since the weak $*$-topology in $\mathscr{S}(A_1\dot{\otimes}\ldots\dot{\otimes}A_n)$ is the initial topology generated by the evaluation maps
\beq
\mathscr{S}\left(A_1\dot{\otimes}\ldots\dot{\otimes}A_n\right)\lra\R,\qquad \phi\mapsto\phi(a_1\otimes\ldots\otimes a_n),\qquad a_i\in A_i\:,
\eeq
it is sufficient to check that all the compositions
\beq
\mathscr{S}(A_1)\times\ldots\times\mathscr{S}(A_n)\lra\mathscr{S}\left(A_1\dot{\otimes}\ldots\dot{\otimes}A_n\right)\lra\R,\qquad(\phi_i)_i\mapsto\phi_1\otimes\ldots\otimes\phi_n\mapsto\prod_i\phi_i(a_i)
\eeq
for fixed $a_i\in A_i$ are continuous. But since this composition factors into two other continuous maps as
\beq
\mathscr{S}(A_1)\times\ldots\times\mathscr{S}(A_n)\lra\R^n\stackrel{\cdot}{\lra}\R,\qquad (\phi_i)_i\mapsto\left(\phi_i(a_i)\right)_i\mapsto\prod_i\phi_i(a_i)
\eeq
it is itself also continuous.
\end{enumerate}
\end{proof}

\begin{defn}[separable and entangled states]
 A state $\phi\in\mathscr{S}\left(A_1\dot{\otimes}\ldots\dot{\otimes}A_n\right)$ is called \emph{separable} if it lies in the closed convex hull of product states, i.e.~if it can be approximated by states of the form
\beqn
\sum_{j=1}^m\lambda_j\left(\phi_1^j\otimes\ldots\otimes\phi_n^j\right)
\eeqn
with coefficients $\lambda_1,\ldots,\lambda_m\geq 0$ and $\sum_j\lambda_j=1$, and $\phi_i^j\in\mathscr{S}(A_i)$. If a state is not separable, it is called \emph{entangled}.
\end{defn}

For the case of matrix algebras, this reduces to the pioneering definitions of Werner~\cite{W}. However, for $A_i=\mathcal{B}(\mathcal{H})$ with an infinite-dimensional Hilbert space $\mathcal{H}$, it differs from the standard one~\cite{ESP}, where the closure is taken with respect to the topology induced by the trace norm on density matrices.

It follows from the definition that the set of separable states does not depend on the choice of tensor product, in the sense that all separable states factor over the minimal tensor product. In other words, a state which does not factor over the minimal tensor product is automatically entangled.

\begin{rem}
It may well be the case that one or several of the $A_i$ above are themselves defined as tensor products of $C^*$-algebras. In order to keep the notation from becoming ambiguous in such a case, we will indicate the $A_i$ by square brackets: e.g.~in
\beq
\left[A\dot{\otimes}B\right]\dot{\otimes}C\:,
\eeq
the tensor product $A\dot{\otimes}B$ is to be considered as a single $C^*$-algebra, so that the above concepts are to be applied with respect to the decomposition of $[A\dot{\otimes}B]\dot{\otimes}C$ into $A\dot{\otimes}B$ and $C$.
\end{rem}

\begin{lem}
\label{tritobi}
Suppose that $\phi\in\mathscr{S}(A\dot{\otimes}B\dot{\otimes}C)$ is such that $\phi_{|A\dot{\otimes}B}$ is entangled. Then $\phi$ is entangled on $A\dot{\otimes}\left[B\dot{\otimes}C\right]$.
\end{lem}

\begin{proof}
It follows directly from the definitions and lemma~\ref{cont} that the separability of $\phi$ on $A\dot{\otimes}\left[B\dot{\otimes}C\right]$ implies the separability of $\phi_{|A\dot{\otimes}B}$.
\end{proof}

The first part of the following lemma is a simple criterion for entanglement which allows us to produce many entangled states. Any state which is not pure is also called \emph{mixed}.

\begin{lem}
\label{puresep}
Let $\phi\in\mathscr{S}(A\dot{\otimes}B)$ be any state.
\begin{enumerate}
\item If $\phi$ is pure and $\phi_{|A}$ is mixed, then $\phi$ is entangled.
\item If the reduced state $\phi_{|A}$ is pure, then $\phi=\phi_{|A}\otimes\phi_{|B}$. In particular, $\phi$ is separable.
\end{enumerate}
\end{lem}

\begin{proof}
\begin{enumerate}
\item It will be shown that when $\phi$ is separable and pure, then $\phi_{|A}$ is pure. To this end, note that by definition, the set of separable states is the closed convex hull of the pure product states and is compact in the weak $*$-topology. Hence by Milman's converse to the Krein-Milman theorem, any pure separable state lies in the closure of the set of pure product states. But by lemma~\ref{cont} and compactness, the set of product states is itself closed; hence $\phi=\phi_{|A}\otimes\phi_{|B}$. But then again by purity of $\phi$, the reduced state $\phi_{|A}$ itself is pure.
\item This proof is an adapted version of an argument made by Wilce~\cite[3.3]{Wil}.

Let $b\in B$ be some fixed element with $0\leq b\leq\mathbbm{1}$ such that $0<\phi_{|B}(b)<1$. The decomposition
\begin{align*}
\phi_{|A}(a)&=\phi(a\otimes\mathbbm{1})=\phi(a\otimes b)+\phi\left(a\otimes(\mathbbm{1}-b)\right)\\[.3cm]
&=\phi_{|B}(b)\frac{\phi(a\otimes b)}{\phi_{|B}(b)}+\phi_{|B}(\mathbbm{1}-b)\frac{\phi\left(a\otimes(\mathbbm{1}-b)\right)}{\phi_{|B}\left(\mathbbm{1}-b\right)}
\end{align*}
writes the state $\phi_{|A}$ as a convex combination of the two states
\beq
a\mapsto \frac{\phi(a\otimes b)}{\phi_{|B}(b)}\quad\textrm{ and }\quad a\mapsto\frac{\phi\left(a\otimes(\mathbbm{1}-b)\right)}{\phi_{|B}\left(\mathbbm{1}-b\right)}
\eeq
However since $\phi_{|A}$ is assumed to be pure, we know that both of these states coincide with $\phi_{|A}$ itself. From the equality of the first of these two states with $\phi_{|A}$, we therefore obtain in particular
\beqn
\label{prod}
\phi(a\otimes b)=\phi_{|A}(a)\phi_{|B}(b)\:.
\eeqn
While this equation has been derived under the assumption that $0<\phi_{|B}(b)<1$, it holds automatically when $\phi_{|B}(b)=0$: for then, considering w.l.o.g. $0\leq a\leq\mathbbm{1}$ means that $0\leq a\otimes b\leq\mathbbm{1}\otimes b$, and therefore $0\leq\phi(a\otimes b)=\phi_{|B}(b)=0$. The case $\phi_{|B}(b)=1$ on the other hand can be reduced to the case $\phi_{|B}(b)=0$ by considering $\mathbbm{1}-b$ instead of $b$. Hence, the validity of~(\ref{prod}) has been established for any $a$ and any $b$ with $0\leq b\leq 1$. But since the linear hull of elementary tensors $a\otimes b$ for such $a$ and $b$ is dense in $A\dot{\otimes}B$, we conclude that $\phi=\phi_{|A}\phi_{|B}$.
\end{enumerate}
\end{proof}

\begin{prop}
\label{ent}
\begin{enumerate}
\item If $A$ or $B$ is commutative, then all states on $A\dot{\otimes} B$ are separable\footnote{Here it is for which tensor product $\dot{\otimes}$ exactly stands for, since the $C^*$-tensor product is unique when $A$ or $B$ is commutative~\cite[IV.4.18]{Tak},~\cite{KR2}.}.
\item If neither $A$ nor $B$ is commutative, then there are entangled states on any $C^*$-tensor product $A\dot{\otimes}B$.
\item Let $A$ be noncommutative. Then there is an entangled state on $\alpha\in\mathscr{S}(M_2\otimes A)$ such that $\alpha_{|M_2}=\tr_2$, the normalized trace.
\end{enumerate}
\end{prop}

\begin{proof}
\begin{enumerate}
\item It is enough to show that every pure state on $A\dot{\otimes}B=A\otimes_{\min}B$ is a product state. But this is well-known~\cite[IV.4.14]{Tak}.

\item It is enough to show this for the minimal tensor product $A\otimes_{\min}B$.

Every noncommutative $C^*$-algebra has an irreducible representation of dimension at least $2$. Let $\pi_A:A\ra\B(\H_A)$ and $\pi_B:B\ra\B(\H_B)$ be such representations of $A$ and $B$, respectively. Then the tensor product representation
\beq
\pi_A\otimes\pi_B:A\otimes_{\min}B\lra\B(\H_A\otimes\H_B)
\eeq
is also irreducible~\cite[IV.4.13]{Tak}. Choose a pair of orthogonal unit vectors $\xi_1,\xi_2\in\H_A$, and a pair of orthogonal unit vectors $\zeta_1,\zeta_2\in\H_B$. Now consider the pure state $\phi\in\mathscr{S}(A\otimes_{\min}B)$ associated to the unit vector
\beq
\frac{\xi_1\otimes\zeta_1+\xi_2\otimes\zeta_2}{\sqrt{2}}\:.
\eeq
A direct calculation shows that the reduced state $\phi_{|A}$ is given by
\beq
\phi_{|A}(a)=\frac{1}{2}\langle\xi_1,\pi_A(a)\xi_1\rangle+\frac{1}{2}\langle\xi_2,\pi_A(a)\xi_2\rangle\:.
\eeq
Now the assertion follows from lemma~\ref{puresep}(i).
\item In the proof of part (ii), take $B=M_2$, let $\pi_B$ be the standard representation on $\C^2$, and set $\zeta_i=e_i$ to be the standard basis.
\end{enumerate}
\end{proof}

\section{Main results}
\label{sec5}

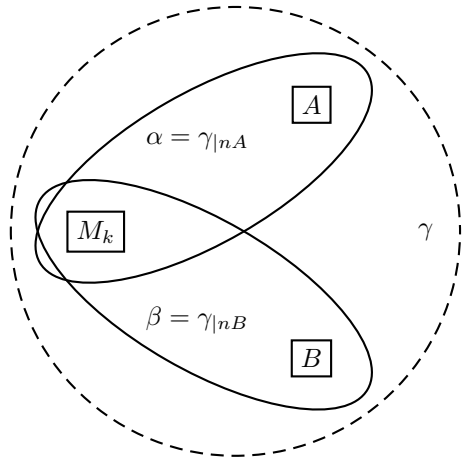
\begin{figure}
\begin{center}
\begin{pspicture}(0,-1)(2,2)
\psset{xunit=4pc}
\psset{yunit=4pc}
\pscircle[linestyle=dashed](1.1,1){3}
\rput{60}(.85,.5){\psellipse(0,0)(.6,1.5)}
\rput(0,1){\psframebox{$M_k$}}
\rput(1.7,2){\psframebox{$A$}}
\rput(1.7,0){\psframebox{$B$}}
\rput{-60}(.85,1.5){\psellipse(0,0)(.6,1.5)}
\rput(2.6,1){$\gamma$}
\rput(.8,1.7){$\alpha=\gamma_{|nA}$}
\rput(.8,0.3){$\beta=\gamma_{|nB}$}
\end{pspicture}
\end{center}
\caption{Illustration of a pair $(\alpha,\beta)$ of $\dot{\otimes}$-compatible states. We think of $A$ and $B$ as observable algebras of physical quantum systems. $M_k$ is the observable algebra of a quantum-mechanical $k$-state system. $\mathscr{S}(M_k(A\dot{\otimes}B))$ is the set of states of the total system.}
\end{figure}

\begin{defn}
We say that a pair of states $\alpha\in \mathscr{S}(M_k(A))$ and $\beta\in \mathscr{S}(M_k(B))$ is \emph{$\dot{\otimes}$-compatible} whenever there is a state $\gamma\in\mathscr{S}(M_k(A\dot{\otimes}B))$ such that
\beq
\gamma_{|M_k(A)}=\alpha,\qquad\gamma_{|M_k(B)}=\beta\:.
\eeq
\end{defn}

Clearly, for a pair of states to be $\dot{\otimes}$-compatible, it needs to be compatible. Similarly, the diagram~(\ref{basic}) yields the basic implications
\begin{center}
$\otimes_{\min}$-compatible $\Longrightarrow$ $\otimes_{\max}$-compatible $\Longrightarrow$ compatible .
\end{center}

\begin{prop}
\label{oncomp}
If both $A$ and $B$ are noncommutative and $A\dot{\otimes}B$ is any $C^*$-tensor product, then there is a pair of states which is compatible, but not $\dot{\otimes}$-compatible.
\end{prop}

\begin{proof}
By proposition~\ref{ent}(iii), there is some entangled state $\alpha$ on $M_2(A)=M_2\otimes A$ with $\alpha_{|M_2}=\tr_2$; and likewise, there is some pure state $\beta$ on $M_2\otimes B$ also reducing to $\tr_2$ on $M_2$. Hence these two states are compatible. However by lemma~\ref{tritobi}, any state $\gamma$ reducing to $\alpha$ has to be entangled with respect to the decomposition $A\dot{\otimes}\left[M_k\otimes B\right]$. However, this contradicts the purity of $\beta$ by lemma~\ref{puresep}(ii).
\end{proof}

This proof is a manifestation of the \emph{monogamy of entanglement} known from quantum entanglement theory: qualitatively, if one system is highly entangled with a second system, then it cannot be highly correlated with a third system. In the case of matrix algebras, quantitative versions of this statement have been derived in~\cite{KW}. Such considerations should in principle also facilitate a quantitative comparison of the set of $\dot{\otimes}$-compatible pairs with the set of compatible pairs, at least when both $A$ and $B$ are matrix algebras. However at the present stage, all results remain at a purely qualitative level.

\begin{prop}
For any $a+b\in M_k(A\dot{\oplus}B)$, the following are equivalent:
\begin{enumerate}
\item $a+b\geq 0$ in $M_k(A\dot{\otimes}B)$.
\item For every pair of states $\alpha\in\mathscr{S}(M_k(A))$, $\beta\in\mathscr{S}(M_k(B))$ which are $\dot{\otimes}$-compatible, it holds that $\alpha(a)+\beta(b)\geq 0$.
\end{enumerate}
\end{prop}

\begin{proof}
Assume (i). Then given any $\dot{\otimes}$-compatible pair $(\alpha,\beta)$, there is a state $\gamma\in\mathscr{S}\left(M_k(A\dot{\otimes}B)\right)$ which extends these two. Hence,
\beq
0\leq \gamma(a+b)=\alpha(a)+\beta(b)\:,
\eeq
and therefore (ii).

Given (ii), it is clear that $\gamma(a+b)\geq 0$ for any state $\gamma\in\mathscr{S}\left(M_k(A\dot{\otimes}B)\right)$, and therefore (i).
\end{proof}

\begin{cor}
$A\otimes_{\min}B=A\otimes_{\max}B$ if and only if
\begin{center}
$\otimes_{\min}$-compatible $\Longleftrightarrow$ $\otimes_{\max}$-compatible\:.
\end{center}
\end{cor}

\begin{proof}
This follows from the previous proposition in combination with proposition~\ref{onlyplus}.
\end{proof}

\paragraph{A non-complete positive isomorphism.}

\begin{prop}
\label{main}
The natural maps
\beq
A\oplus_\ast B\lra A\oplus_{\max}B\lra A\oplus_{\min}B 
\eeq
are isomorphisms of ordered vector spaces. However when $A$ and $B$ are both noncommutative, the first map is not an isomorphism of operator system: there are elements in $M_2(A\oplus_{\max}B)$ which are not positive in $M_2(A\oplus_\ast B)$.
\end{prop}

\begin{proof}
It has already been noted that both maps are ucp by~(\ref{basic}). Hence for the second assertion, we have to check that the positive cone on $A\oplus_{\min}B$ is not bigger than the one on $A\oplus_\ast B$.

So suppose that $a+b\geq 0$ in $A\oplus_{\min}B$. It has to be shown that also $a+b\geq 0$ in $A\oplus_\ast B$. To this end, it is sufficient by proposition~\ref{freepos} to check that $\alpha(a)+\beta(b)\geq 0$ for any states $\alpha\in\mathscr{S}(A)$ and $\beta\in\mathscr{S}(B)$. But this is clear by assumption, since on $A\otimes_{\min}B$ we have that
\beq
\alpha(a)+\beta(b)=(\alpha\otimes\beta)(a\otimes\mathbbm{1}+\mathbbm{1}\otimes b)\geq 0.
\eeq

For the second assertion, we assume to the contrary that it would indeed be a complete isomorphism. This would mean that
\beq
\mathscr{S}\left(M_2(A\oplus_{\ast}B)\right)=\mathscr{S}\left(M_2(A\oplus_{\max}B)\right) ,
\eeq
implying that a compatible pair of states is automatically $\otimes_{\max}$-compatible. However by proposition~\ref{oncomp}, we know that this is not the case when both $A$ and $B$ are noncommutative.
\end{proof}

\paragraph{Drawing on symmetries.} In this subsection, we consider the operator systems $A\oplus_\ast B$, $A\oplus_{\max}B$ and $A\oplus_{\min}B$ as operator spaces. Recall that any operator system inherits an operator space structure, for example by embedding it into some $\B(\H)$ and noting that the matricial norms do not depend on the choice of embedding. Since our operator systems have been defined as subspaces of $C^*$-algebras, they directly come equipped with their natural operator space structures.

In cases where the $C^*$-algebras under examination have a high degree of symmetry, there is a tight comparison that can be done between the norm $||\cdot||_\ast$ on $A\oplus_*B$ induced from the unital free product $A\ast_1 B$ and the norm $||\cdot||_{\min}$ induced from the $C^*$-algebra tensor product $A\otimes_{\min}B$.

We start by outlining the general idea with normed spaces. Given normed spaces $E$ and $F$, there are several natural ways to equip the direct sum $E\oplus F$ with a norm, with the most common ones being the following:
\beq
||e+f||\equiv\max\left\{||e||,||f||\right\},\quad\textrm{or}\quad||e+f||\equiv\sqrt{||e||^2+||f||^2},\quad\textrm{or}\quad||e+f||\equiv||e||+||f||\:.
\eeq
All of these satisfy the important property that switching the sign of either component $e$ or $f$ does not change the value of the norm. The following simple proprosition highlights the importance and utility of this property for the comparison of norms.

\begin{prop}
\label{symprop}
Let $E$ and $F$ be  normed spaces and $||\cdot||_{\mathrm{sym}}$ a norm on $E\oplus F$ which restricts to the given norms on $E$ and $F$. If $||\cdot||_{\mathrm{sym}}$ satsfies the symmetry condition
\beqn
\label{symcond}
||e-f||_{\mathrm{sym}}=||e+f||_{\mathrm{sym}}\qquad\forall\: e\in E,\: f\in F\:,
\eeqn
and $||\cdot||_{\mathrm{any}}$ is any other norm on $E\oplus F$ which also extends the given norms on $E$ and $F$, then
\beq
||e+f||_{\mathrm{any}}\leq 2||e+f||_{\mathrm{sym}}\:.
\eeq
\end{prop}

In particular, all direct sum norms with the symmetry property are equivalent and differ by a factor of at most $2$.

\begin{proof}
\begin{align*}
||e+f||_{\mathrm{any}}&\leq||e||_E+||f||_F\\\\
&=\left|\left|\frac{1}{2}\left(e+f\right)+\frac{1}{2}\left(e-f\right)\right|\right|_E+\left|\left|\frac{1}{2}\left(e+f\right)+\frac{1}{2}\left(-e+f\right)\right|\right|_F\\\\
&\leq ||e+f||_{\mathrm{sym}}+||e-f||_{\mathrm{sym}}=2||e-f||_{\mathrm{sym}}
\end{align*}
\end{proof}

\begin{defn}[$\Z_2$-graded $C^*$-algebra]
We call $A$, a unital $C^*$-algebra, \emph{$\Z_2$-graded} if it comes equipped with a direct sum decomposition
\beq
A=A_0\oplus A_1,\qquad A_0^*=A_0, \quad A_1^*=A_1,\quad \mathbbm{1}\in A_0,
\eeq
\beq
A_i\cdot A_j\subseteq A_{(i+j)\:\mathrm{mod}\:2}\:.
\eeq
\end{defn}

Every such decomposition defines an automorphism $\sigma:A\ra A$ by setting $\sigma_{|A_0}\equiv\id_{A_0}$ and $\sigma_{|A_1}\equiv -\id_{A_1}$. This automorphism has order $2$. Conversely, every automorphism of order $2$ yields a splitting of $A$ into corresponding subspaces $A_0$ and $A_1$, thereby defining a $\Z_2$-grading.

\begin{ex}
Let $G=\langle\,\Gamma\,|\,R\,\rangle$ be a discrete group with generators $\Gamma$ and relations $R$ such that every generator occurs in each relation an even number of times. ($0$ is even.) Then $C^*(G)$ (or $C^*_r(G)$) is $\Z_2$-graded by splitting it into the closed linear span of reduced words of even length and the closed linear span of reduced words of odd length. In particular, this applies to $G=\F_n$ for any finite or infinite cardinal $n$.
\end{ex}

If we do not consider all elements of $A$ for taking the unital direct sum, but restrict to the subspace of odd ones, then the simple consideration of proposition~\ref{symprop} gives a strong bound:

\begin{prop}
\label{gradmain}
Let $A$ be a unital $\Z_2$-graded $C^*$-algebras and $B$ any unital $C^*$-algebra. Then for $a+b\in M_k(A\oplus_1 B)$ with $a\in M_k(A_1)$ an odd element, we have
\beq
||a+b||_{\min}\leq||a+b||_{\max}\leq||a+b||_{\ast}\leq 2||a+b||_{\min}\:.
\eeq
In particular, all three operator space structures on $A_1\oplus B\subseteq A\oplus_1 B$ are completely isomorphic.
\end{prop}

\begin{proof}
Since the first two inequalities are clear by~(\ref{basic}), it remains to prove that $||a+b||_{\ast}\leq 2||a+b||_{\min}$ for $a\in M_k(A_1)$ and $b\in M_k(B)$. To this end, note first that the sum $A_1+B\subseteq A\oplus_{\min} B$ is actually a direct sum, so that we have a direct sum splitting
\beqn
\label{subspace}
M_k(A_1\oplus B)=M_k(A_1)\oplus M_k(B)
\eeqn
We have induced norms on this vector space coming from its embeddings into $M_k(A\ast_1 B)$ and $M_k(A\otimes_{\min}B)$. Since functoriality shows these norms to be invariant under the grading automorphism on $A$, it follows that both induced norms on~(\ref{subspace}) have the symmetry property~(\ref{symcond}). Hence the assertion follows from proposition~\ref{symprop}.
\end{proof}

\section{Implications for $C^*(\F_n)\otimes C^*(\F_n)$}
\label{qwepsec}

The main motivation for obtaining the previous results was to study properties of $C^*$-tensor norms on $C^*(\F_n)\otimes C^*(\F_n)$. We use similar notation as~\cite{Radu} by setting
\beq
\mathcal{X}=C^*(\F_n)\otimes\mathbbm{1}+\mathbbm{1}\otimes C^*(\F_n)
\eeq
and
\beq
\mathcal{X}_1=C^*(\F_n)_1\otimes\mathbbm{1}+\mathbbm{1}\otimes C^*(\F_n)\:,
\eeq
where $C^*(\F_n)_1\subseteq C^*(\F_n)$ is the closed subspace spanned by the reduced words of odd length. Then our propositions~\ref{main} and~\ref{gradmain} readily yield the following:

\begin{thm}
Let $n\geq 2$ be any cardinal. We consider $\mathcal{X}$ and $\mathcal{X}_1$ as operator subspaces of
\beq
C^*(\F_{2n}),\quad C^*(\F_n)\otimes_{\max}C^*(\F_n),\quad\textrm{and}\quad C^*(\F_n)\otimes_{\min}C^*(\F_n)\:.
\eeq
Then we have,
\begin{enumerate}
\item For $X\in\mathcal{X}$, all three norms coincide:
\beq
||X||_{C^*(\F_{2n})}=||X||_{\max}=||X||_{\min}\:.
\eeq
\item There are $X\in M_2(\mathcal{X})$ for which
\beq
||X||_{C^*(\F_{2n})}>||X||_{\max}\:.
\eeq
\item For $X\in M_k(\mathcal{X}_1)$, all three norms are equivalent:
\beq
||X||_{\min}\leq ||X||_{\max}\leq ||X||_{C^*(\F_{2n})}\leq 2||X||_{\min}\:.
\eeq
\end{enumerate}
\end{thm}

In particular, the last inequality is a significant improvement over R\u{a}dulescu's result~\cite{Radu}, since both the subspace considered is bigger and the bound is tighter.

However, in order to obtain this result, we have made essentially no use of the geometry of the free group, or properties of partially commuting sets of unitaries, or anything else specific to $C^*(\F_n)$. Hence it seems unlikely that the QWEP conjecture could be proven along similar lines. In fact, the same arguments also show the following, in the analogous notation: for any
\beq
X\in M_k\left(C_r^*(\F_n)_1 \otimes\mathbbm{1}+\mathbbm{1}\otimes C_r^*(\F_n)\right)\:,
\eeq
the estimate
\beqn
\label{tc}
||X||_{\min}\leq ||X||_{\max}\leq||X||_{C_r^*(\F_{2n})}\leq 2||X||_{\min}
\eeqn
is valid; however, for the reduced group $C^*$-algebra $C_r^*(\F_n)$, it is well-known (see e.g.~\cite[11.3.14]{KR2} for $n=2$) that
\beq
C_r^*(\F_{n})\otimes_{\min}C_r^*(\F_{n})\neq C_r^*(\F_{n})\otimes_{\max}C_r^*(\F_{n})\:,
\eeq
despite the tight comparison\eq{tc}.

\bibliographystyle{halpha}
\bibliography{comparison_norms}

\end{document}